\documentclass[10pt]{article}
\usepackage[letterpaper,margin=1in]{geometry}
\usepackage{amsmath,amssymb,amsthm}
\newtheorem{proposition}{Proposition}
\DeclareMathOperator{\Spec}{Spec}
\title{An endpoint exception to a conjectured equality characterization\\
for Laplacian eigenvalue products}
\author{Scott Gibson}
\date{September 6, 2026}
\begin{document}
\maketitle
\begin{abstract}
We give eight-vertex counterexamples to the equality characterization in Conjecture 20 of Chen, Guo,
Li and Wang, Electron. J. Combin. 32(4) (2025), P4.37. The examples occur at the included endpoint
$k=3n/4$ and extend to infinite families, including a family for which both the graph and its complement
are connected. These examples do not contradict the numerical inequality in that conjecture.
\end{abstract}

Let $G$ be a finite simple graph of order $n$, and write the eigenvalues of its Laplacian $L(G)=D(G)-A(G)$
as $\mu_1(G)\ge\cdots\ge\mu_n(G)=0$. Conjecture 20 of \cite{chen}, on printed page 11, asserts that
\begin{equation}\label{ineq}
\mu_k(G)\mu_k(\overline G)\le n(n-k),\qquad 1\le k\le 3n/4,
\end{equation}
and characterizes equality by requiring $G$ or $\overline G$ to be a join $K_k\vee H$, where $H$ has $n-k$ vertices and at least
$k+1$ connected components. Here $\vee$ denotes the graph join. The following elementary examples contradict
the necessary direction of that characterization. The graph families and spectral ingredients are standard;
in particular, the barbell characteristic-polynomial calculation appears in the proof of Lemma 20 of Li and
Guo \cite{li}. The point here is their application to the endpoint equality claim in \cite{chen}.

\begin{proposition}
For every integer $m\ge2$, the graph $G=K_{2m,2m}$ attains equality in \eqref{ineq} at $n=4m$, $k=3m$,
although neither $G$ nor $\overline G$ has the prescribed form.
\end{proposition}
\begin{proof}
Put $s=2m$. The complement of $K_{s,s}$ is $K_s\cup K_s$. Their Laplacian spectra are
\[
\Spec_L(K_{s,s})=\{2s,s^{[2s-2]},0\},\qquad
\Spec_L(K_s\cup K_s)=\{s^{[2s-2]},0^{[2]}\},
\]
where brackets denote multiplicity. Since $3m\le4m-2$, both $k$th eigenvalues equal $s$, and
\[
\mu_k(G)\mu_k(\overline G)=s^2=4m^2=n(n-k).
\]
Every vertex of the $K_k$ part of $K_k\vee H$ is universal. The degrees in $G$ and $\overline G$ are respectively $s$ and $s-1$,
both strictly less than $n-1=2s-1$. Thus neither graph is of the required form.
\end{proof}

In particular, $K_{4,4}$ gives an eight-vertex counterexample at $k=6$:
\[
\mu_6(K_{4,4})\mu_6(K_4\cup K_4)=4\cdot4=16=8(8-6).
\]
The counterexample persists when connectivity is required for both graphs.

\begin{proposition}
For every integer $m\ge2$, let $F$ be obtained from $K_{2m,2m}$ by deleting one edge. Then $F$ and
$\overline F$ are connected, attain equality in \eqref{ineq} at $n=4m$, $k=3m$, and neither has the prescribed form.
\end{proposition}
\begin{proof}
Again put $s=2m$. The graph $B=\overline F$ consists of two copies of $K_s$ joined by one edge. Its characteristic
polynomial is the equal-clique specialization of the polynomial calculation in \cite[proof of Lemma 20]{li};
the lemma's strict inequality assumes unequal cliques, but that determinant calculation does not. For completeness,
we give a direct derivation. Partition the vertices into the first bridge endpoint, the other $s-1$
vertices of its clique, the second endpoint, and the other $s-1$ vertices of the second clique.

Vectors supported on either set of non-endpoints and summing to zero give the eigenvalue $s$ with total
multiplicity $2s-4$. On vectors constant on each of the four cells, the Laplacian has quotient matrix
\[
Q=\begin{pmatrix}
s&-(s-1)&-1&0\\
-1&1&0&0\\
-1&0&s&-(s-1)\\
0&0&-1&1
\end{pmatrix}.
\]
The symmetric and antisymmetric subspaces under exchange of the two cliques reduce $Q$, respectively, to
\[
\begin{pmatrix}s-1&-(s-1)\\-1&1\end{pmatrix},\qquad
\begin{pmatrix}s+1&-(s-1)\\-1&1\end{pmatrix}.
\]
Their characteristic polynomials are $t(t-s)$ and $t^2-(s+2)t+2$. With $\Delta=s^2+4s-4$, this gives
\[
\Spec_L(B)=\left\{\frac{s+2+\sqrt\Delta}{2},\ s^{[2s-3]},\ \frac{s+2-\sqrt\Delta}{2},\ 0\right\}.
\]
Using $L(F)+L(B)=2sI-J$ on the subspace orthogonal to the all-ones vector gives
\[
\Spec_L(F)=\left\{\frac{3s-2+\sqrt\Delta}{2},\ s^{[2s-3]},\ \frac{3s-2-\sqrt\Delta}{2},\ 0\right\}.
\]
Both displays are in descending order. Indeed, for $s\ge4$,
\[
(s-2)^2<\Delta<(s+2)^2,\qquad \Delta<(3s-2)^2.
\]
Thus the eigenvalue $s$ occupies positions $2$ through $2s-2$ in both spectra. As $k=3m\le4m-2=2s-2$,
the product is again $s^2=n(n-k)$. Both graphs are connected, and their maximum degrees are $s<2s-1$,
so neither has a universal vertex.
\end{proof}

For the eight-vertex connected pair, the characteristic polynomials simplify to
\begin{align*}
\det(tI-L(F))&=t(t-4)^5(t^2-10t+18),\\
\det(tI-L(\overline F))&=t(t-4)^5(t^2-6t+2).
\end{align*}

\paragraph{Scope.}
Theorem 18 of \cite{chen} already proves \eqref{ineq} for $n/2\le k\le3n/4$. For $k\ge n/2$, eigenvalue ordering and
the complement identity imply
\[
\mu_k(G)+\mu_k(\overline G)\le n,\qquad
\mu_k(G)\mu_k(\overline G)\le n^2/4.
\]
Consequently equality in \eqref{ineq} is impossible when $n/2\le k<3n/4$. At $k=3n/4$, it holds precisely when both
$k$th eigenvalues are $n/2$. The examples above expose a missing endpoint case in the structural characterization;
they leave the inequality for $k<n/2$ unresolved.

\paragraph{Computational verification and AI assistance.}
ChatGPT (OpenAI) was used in finding the examples,
developing the arguments, and drafting this note. The eight-vertex characteristic polynomials were checked
using SymPy and a separate exact-rational determinant verifier, supplied as an ancillary file. The proofs
above are self-contained and do not require numerical eigenvalue calculations.

\end{document}